\documentclass{article}%
\usepackage{amssymb}
\usepackage{graphicx}
\usepackage{amsmath}
\usepackage{amsfonts}%
\usepackage[usenames,dvipsnames]{pstricks}
\usepackage{epsfig}
\usepackage{pst-grad} 
\usepackage{pst-plot} 
\usepackage{caption}

\providecommand{\U}[1]{\protect\rule{.1in}{.1in}}
\newtheorem{theorem}{Theorem}

\newtheorem{claim}{Claim}

\newtheorem{ob}[theorem]{Observation}
\newtheorem{proposition}[theorem]{Proposition}

\newenvironment{proof}[1][Proof]{\textbf{#1.} }{\ \rule{0.5em}{0.5em}}
\date{}
\begin{document}

\title{On the global offensive alliance in unicycle graphs
}

\author{\textbf{Mohamed Bouzefrane}\\{\small Dept of Mathematics,} \ {\small B.P. 270, University of Blida,
Algeria}\\mohamedbouzefrane@gmail.com, \and
\and
\textbf{Saliha Ouatiki\thanks{The corresponding author.}}\\{\small L'IFORCE laboratory, (U.S.T.H.B)} \& {\footnotesize Dept of Mathematics, University of Boumerdes, Algeria}
\\saliha\_ouatiki@yahoo.fr}
\maketitle

\begin{abstract}
For a graph $G=(V,E)$, a set $S\subseteq V$ is a dominating set if every
vertex in $V-S$ has at least a neighbor in $S$. A dominating set $S$ is a global
offensive alliance if for each vertex $v$ in $V-S$ at least half the vertices
from the closed neighborhood of $v$ are in $S.$ The domination number $\gamma(G)$ is the minimum cardinality of a dominating set of $G$, and the global
offensive alliance number $\gamma_{o}(G)$ is the minimum cardinality of a
global offensive alliance of $G$. We show that if $G$ is a connected unicycle graph
of order $n$ with $l(G)$ leaves and $s(G)$ support vertices then $\gamma_{o}(G)\geq\frac{n-l(G)+s(G)}{3}$.
Moreover, we characterize all extremal unicycle graphs attaining this bound.

\textbf{Keywords:} domination, global offensive alliance, unicycle graph.

\end{abstract}

\section{Introduction}

Let $G=(V,E)$ be a finite and simple graph of order $n$. The \emph{open
neighborhood} of a vertex $v$ is $N(v)=\left\{u\in V/uv\in E\right\}$ and
the \emph{closed neighborhood} of $v$ is $N\left[ v\right]= N(v)\cup\left\{v\right\}$. If $S\subset V$, then $N(S)=\cup_{v\in S} N(v)$ , $N\left[
S\right]= N(S)\cup S$ and the subgraph induced by $S$ in $G$ is denoted
$G\left[ S\right]$. The \emph{degree} of $v$, denoted by $deg_{G}(v),$ is the size
of its open neighborhood. A vertex of degree one is called a \emph{pendent
vertex }or a \emph{leaf} and its neighbor is called a \emph{support vertex}. If $v$ is
a support vertex, then $L_{v}$ will denote the set of the leaves attached at
$v$. We denote the set of leaves of a graph $G$ by $L(G)$ and the set of
support vertices by $S(G)$, and let $\left\vert L(G)\right\vert =l(G)$,
$\left\vert S(G)\right\vert=s(G)$.

For a graph $G=(V,E)$, a set of vertices $S$ is a \emph{dominating set} if every
vertex in $V-S$ has at least a neighbor in $S$. The \emph{domination number}
$\gamma(G)$ is the minimum cardinality of a dominating set of $G$. For a comprehensive treatment of domination in graphs, see the
books of Haynes, Hedetniemi and Slater \cite{4,5}.

In \cite{6} Hedetniemi, Hedetniemi and Kristiansen introduced several types of
alliances in graphs, including the offensive alliance we consider here. A
dominating set $S$ of $G$ is called a \emph{global offensive alliance} if for every
vertex $v\in V-S$, $|N[v]\cap S|\geq |N[v]-S|$. The \emph{global offensive alliance
number} $\gamma_{o}(G)$ is the minimum cardinality of a global offensive
alliance. The entire vertex set is a global offensive alliance for any graph
$G$, so every graph $G$ has a global offensive alliance number. We abbreviate
global offensive alliance as \emph{GOA}. A GOA with minimum cardinality
$\gamma_{o}(G)$ is called $\gamma_{o}(G)$-set.
A graph $G$ is a \emph{unicycle graph} if it owns only one cycle.

Offensive alliances in graphs were be studied in \cite{1,2,3}.
In this paper we give a lower bound on the global offensive alliance number. More
precisely we show that every connected unicycle graph $G$ of order $n$ with $l(G)$
leaves and $s(G)$ support vertices satisfies $\ \gamma_{o}(G)\geq
$\ \ $(n-l(G)+s(G))/3$\ and we characterize all extremal unicycle graphs attaining this lower bound.\ \ \ \ \ \ \ \ \ \ \ \ \

\section{Mains results}

We begin by this following straightforward observation.

\begin{ob}
\label{obs00}If $G$ is a connected graph of order at least three, then there
is a $\gamma_{o}(G)$-set that contains all the support vertices.
\end{ob}

Bouzefrane and Chellali \cite{1} gave a lower bound of the global offensive alliance of trees
and characterized all extremal trees attaining this bound by considering a family $\mathcal{F}$
of trees of order at least three that can be obtained from $r$ disjoint stars by first adding $r-1$
edges so that they are incident only with centers of the stars and the resulting graph is connected,
and then subdividing each new edge exactly once. They prove the following result.

\begin{theorem}
[\cite{1}]\label{thm1} Let $T$ be a tree of order $n\geq3$ with $l(T)$ leaves and
$s(T)$ support vertices. Then
\[
\gamma_{o}(T)\geq\frac{n-l(T)+s(T)+1}{3},
\]
with equality if and only if $T\in\mathcal{F}.$
\end{theorem}

In the next theorem, we give a lower bound on the global offensive alliance of a connected unicycle graph.

\begin{theorem}\label{thm2}
Let $G$ be a connected unicycle graph of order $n$ with $l(G)$ leaves and $s(G)$ support vertices. Then $\gamma_{o}(G)\geq\frac{n-l(G)+s(G)}{3}$.
\end{theorem}

\bigskip

\begin{proof} Let $G$ be a unicycle graph of order $n$ and
cycle $C$. It's clear that $n\geq3.$ We proceed by induction on the order of
$G$. It is easy to check that the result is valid for $n=3$ or $4$. We suppose
that every unicycle graph $G^{\prime}$ of order $n^{\prime}< n$ with $l^{\prime}$
leaves and $s^{\prime}$ support vertices satisfies $\gamma_{o}(G^{\prime})\geq\frac{n^{\prime}-l^{\prime}+s^{\prime}}{3}.$ Let
$D$ be a $\gamma_{o}(G)$-set. By Observation \ref{obs00}, we can assume that
$D$ contains all support vertices of $G$. If $G$ is a cycle\ $C_{n}$, then the
statement is true because $\gamma_{o}(C_{n})=\left\lceil \frac{n}%
{2}\right\rceil >\frac{n}{3}.$ \ Assume now that $G\neq C_{n}$. Let $x,y,z$ be
three consecutive vertices on $C$ in this order. Let us now examine the
following cases.

\textbf{Case 1. }$x\in D$ and $y\notin D.$ Then $y$ is not a support vertex.

\textbf{Case 1.1. }$\left\vert V(C)\right\vert =3.$

If $z\notin D$, then $y,z$ have degree at least $3$, otherwise $D$ is not a
\emph{GOA} of $G$. In this case, $D$ is also a \emph{GOA} of $H=G-yz$ implying
that $\gamma_{o}(G)\geq\gamma_{o}(H)$. Since $H$ is a tree of order $n(H)=n(G)=n$ with $l(H)=l(G)$ and
$s(H)=s(G),$ it follows by the Theorem \ref{thm1} that
\begin{align*}
\gamma_{o}(G) \geq\gamma_{o}(H)&  \geq\frac{n(H)-l(H)+s(H)+1}{3}\\
&  =\frac{n-l(G)+s(G)+1}{3}\\
&  >\frac{n-l(G)+s(G)}{3}.
\end{align*}

Suppose now $z\in D$ and assume first that $d(x)=d(z)=2.$ Therefore
$d(y)\geq3$ because $G$ is not a cycle. Let $D^{\prime}=(D-\{z\})\cup\{y\}.$ It easy to see that $D^{\prime}$ is \emph{GOA} of $H=G-yz$
and so $\gamma_{o}(G)\geq\gamma_{o}(H)$. Also, it is not difficult to verify
that $n(H)=n(G)=n$, $l(H)=l(G)+1$ and $s(H)=s(G)+1.$ As $H$ is a tree, Theorem
\ref{thm1} implies that
\begin{align*}
\gamma_{o}(H)  &  \geq\frac{n(H)-l(H)+s(H)+1}{3}\\
&  =\frac{n-l(G)-1+s(G)+1+1}{3}=\frac{n-l(G)+s(G)+1}{3},
\end{align*}
which gives $\gamma_{o}(G)>\frac{n-l(G)+s(G)}{3}.$

Assume now that one of $x,z$ say $x$ has degree at least $3$ and let $H=G-xz$.
If $d_{G}(z)=2$, the set $D-\{z\}\cup \{y\}$ is a \emph{GOA} of $H$ then $n(H)=n(G)=n$,
$l(H)=l(G)+1$ and $s(H)=s(G)+1.$ So as in the previous case, $\gamma_{o}%
(G)>\frac{n-l(G)+s(G)}{3}.$ If $d_{G}(z)\geq3$, the set $D$ still a \emph{GOA} of $H$, then $n(H)=n(G)=n$,
$l(H)=l(G)$ and $s(H)=s(G).$ So, by Theorem \ref{thm1}
\begin{align*}
\gamma_{o}(G)\geq \gamma_{o}(H)  &  \geq\frac{n(H)-l(H)+s(H)+1}{3}\\
&  =\frac{n-l(G)+s(G)+1}{3}\\
&  >\frac{n-l(G)+s(G)}{3}.
\end{align*}

\textbf{Case 1.2}. $\left\vert V(C)\right\vert \geq4.$ If $d(y)=2$, then $z\in
D$. Otherwise $D$ is not a \emph{GOA}. Removing $y$ and identify $x$ and $z$ in
one vertex noted $\overline{xz}$. Then $G^{\prime}$ is a unicycle graph of order $n^{\prime}=n-2, l^{\prime}=l(G), s^{\prime}\geq s(G)-1.$
On the other hand $(D-\left\{x,z\right\})\cup\left\{\overline
{xz}\right\}  $ is a \emph{GOA} of $G^{\prime}$. We use induction on
$G^{\prime}$, we obtain $\gamma_{o}(G)\geq\gamma_{o}(G^{\prime})+1\geq
\frac{n^{\prime}-\ell^{\prime}+s^{\prime}}{3}+1\geq\frac{n-2-l(G)+s(G)-1}{3}+1$
and so $\gamma_{o}(G)\geq\frac{n-l(G)+s(G)}{3}$. \\

Now we assume that $d(y)\geq3$. Then $y$ is not a support vertex else by our
choice of $D$, $y\in D$. Let $w$ be the other neighbor of $y$ such that $w\notin
D$. If such vertex does not exist, then $N(y)\subset D$ and in this case
removing the edge $xy$. The resulting graph $G^{\prime}$ is a tree of order
$n^{\prime}=n$ with $l^{\prime}\leq l(G)+1$, $s^{\prime}\geq s(G)$. As $D$ is a \emph{GOA} of $G^{\prime}$, we obtain by
Theorem \ref{thm1} $\gamma_{o}(G)\geq\gamma_{o}(G^{\prime})\geq\frac{n^{\prime}-\ell^{\prime}+s^{\prime}+1}%
{3}\geq\frac{n-\ell(G)+s(G)}{3}$. So we can assume that $w$ exist. In this case
removing the edge $yw$ and note $G^{\prime}$ the resulting graph. If $w=z$, as $z \notin D, y \notin D$ then $deg_{G}(z)\geq 3$. The set
$D$ still a \emph{GOA} of $G^{\prime}$ which is a tree of order $n^{\prime}=n$
with $l^{\prime} = l(G)$, $s^{\prime} = s(G)$. We use Theorem \ref{thm1} on
$G^{\prime}$ we obtain $\gamma_{o}(G)> \frac{n-\ell(G)+s(G)}{3}$.
Now assume that $w\neq z$ . So $z\in D.$ The graph $G^{\prime}$ consists of two connected components, one is a unicycle graph $G^{\prime\prime}$ and the other is a tree $T^{^{\prime}}$ containing $w$. Note that $T^{\prime}$ is of order $n^{\prime}\geq5$ as both $y,w$ are not in $D$, so $w$ has at least two neighbors in $D\cap V(T^{\prime})$. On the other hand $D\cap V(G^{\prime\prime})$ and $D\cap V(T^{\prime})$ are both a global offensive alliances of $G^{\prime\prime}$ and $T^{\prime}$ respectively, hence it follows that $\gamma_{o}(G)\geq\gamma
_{o}(G^{^{\prime}})=\gamma_{o}(G^{\prime\prime})+\gamma_{o}(T^{\prime})$.

Set $n(G^{\prime\prime})=n^{\prime\prime}, l(G^{\prime\prime})=l^{\prime\prime}, s(G^{\prime\prime})=s^{\prime\prime}, l(T^{\prime})=l^{\prime}$ and $s(T^{\prime})=s^{\prime}$.

Note also that since $d(y)\geq3$ and $T^{\prime}$ is a tree of order $n^{\prime}\geq 5,$ then $n=n^{\prime\prime}+n^{\prime}, l(G)=l^{^{\prime\prime}}+l^{\prime}$ and $s(G)= s^{\prime\prime}+s^{\prime}$. Now we use induction on $G^{\prime\prime}$ and by Theorem \ref{thm1} \ on $T^{\prime}$ we obtain
$\gamma_{o}(G)\geq \gamma_{o}(G^{\prime\prime})+\gamma_{o}(T^{\prime})\geq \frac{n^{\prime\prime}-l^{\prime\prime}+s^{\prime\prime}}{3}+\frac{n^{\prime}-l^{\prime}+s^{\prime}+1}{3}
=\frac{n-l(G)+s(G)+1}{3} >\frac{n-l(G)+s(G)}{3}$.

\textbf{Case 2. }$x,y\in D\mathbf{\ .}$ If $d(x)=2$, then we can replace $x$
in $D$ by his second neighbor on $C$. In this case we obtain a new
$\gamma_{o}(G)-$set $D^{\prime}$ such that $x\notin D^{\prime}$, $y\in D^{\prime}$ or this case has been considered (see case $1$). So $d(x)\geq3$
and similarly $d(y)\geq3.$ Let $G^{\prime}$ be the obtained graph after
removing the edge $xy.$ Then $G^{\prime}$ is a tree of order $n^{\prime}$ with
$l^{\prime}=l(G),s^{\prime}=s(G)$. Since $G^{\prime}$ is a tree and $D$ is a
\emph{GOA} of $G^{\prime}$ so we obtain by Theorem \ref{thm1}, $\gamma_{o}(G)>\frac{n-\ell(G)+s(G)}{3}$.

\textbf{Case 3. }$x\notin D,$\textbf{\ }$y\notin D\mathbf{.}$ Then both
$x,y$ must be of degree at least $3$, otherwise $D$ is not a \emph{GOA}. Furthermore, $x,y$ are not a support vertices by the choice of $D$.
Let $H$ be the graph obtained from $G$ by removing the edge $xy$. Clearly, $H$ is a tree of order $n(H)=n$ with $l(H)= l(G)$ and $s(H)=
s(G)$. As $D$ is a \emph{GOA} of $H$, it follows by Theorem \ref{thm1} that $\gamma_{o}(G)>\frac{n-l(G)+s(G)}{3}$. \end{proof}
\bigskip

In order to characterize all unicycle graphs which attain the bound of the Theorem \ref{thm2}, we need the next theorems which give properties of all vertices of a cycle of these graphs. We start with the next one which state that if the lower bound of the Theorem \ref{thm2} is reached, then the unicycle graph is a bipartite one. Chellali \cite{2} proved for a bipartite graph $G$ of order $n$ without isolated vertices that $\gamma_{o}(G)\leq \frac{n-l(G)+s(G)}{2}$. The next result still verify this inequality.

\begin{theorem}\label{them3}
Let $G$ be a connected unicycle graph of order $n$ with cycle $C$, $l(G)$ leaves and $s(G)$ support vertices. If $\gamma_{o}(G)=\frac{n-l(G)+s(G)}{3}$, then $C$ is an even cycle.
\end{theorem}

\begin{proof}
By the absurd: we assume that $G$ is an unicycle graph with an odd cycle $C$ such that $\gamma_{o}(G)=\frac{n-l(G)+s(G)}{3}$.
We have $G\neq C_{n}$ as $\gamma_{o}(C_{n})=\lceil \frac{n}{2}\rceil=\frac{n+1}{2}\neq \frac{n}{3}$. So, $G$ has at least one support vertex. Let $D$ be a $\gamma_{o}(G)$-set and
let $x,y$ be two adjacent vertices of $C$. It follows by Case $3$ of the proof of Theorem \ref{thm2} that at least one of $x, y$ is in $D$. Thus, let us consider the case $x,y\in D$. The Case $2$ of the proof of the previous theorem excluded the case where both $x,y$ are of degree at least $3$. So, assume that $deg_{G}(x)=2$.
If $|V(C)|=3$, then Case 1.1 of the proof of Theorem \ref{thm2} leads to $\gamma_{o}(G)>\frac{n-l(G)+s(G)}{3}$. So assume that $|V(C)|\geq 5$ as $C$ is an odd cycle. Let us identify the vertices $x$ and $y$ on the vertex $\overline{xy}$. The obtained graph $G^{\prime}=(G-\{x,y\})\cup \{\overline{xy}\}$ is an unicycle graph of order $n^{\prime}=n-1$. Since $deg_{G}(x)=2$ then $l(G^{\prime})=l^{\prime}=l(G)$ and $s(G^{\prime})= s^{\prime}=s(G)$. The set $(D-\{x,y\})\cup \{\overline{xy}\}$ is a global offensive alliance of $G^{\prime}$, which implies that $\gamma_{o}(G^{\prime})\leq |D|-1= \gamma_{o}(G)-1= \frac{n-l(G)+s(G)}{3}-1=\frac{n-l(G)+s(G)-3}{3}$.
On the other hand, Theorem \ref{thm2} yields $\gamma_{o}(G^{\prime})\geq \frac{n^{\prime}-l^{\prime}+s^{\prime}}{3}= \frac{n-l(G)+s(G)-1}{3}$. Thus we obtain $\frac{n-l(G)+s(G)-1}{3}\leq \gamma_{o}(G^{\prime})\leq \frac{n-l(G)+s(G)-3}{3}$, which is a contradiction.

We deduce then that for every edge $xy$ of $C$, only one extremity $x$ or $y$ is in $D$ and so either $|V(C)\cap D|= \frac{|V(C)|+1}{2}$ in which case there exists an edge $xy$ such that $x, y\in D$ which is a contradiction, or $|V(C)\cap D|= \frac{|V(C)|-1}{2}$ in which case there exists an edge $xy$ such that both $x,y$ are not in $D$, which is also a contradiction.\end{proof}

\begin{theorem}\label{thm3} Let $G$ be a connected unicycle graph of order $n$ with cycle $C$, $l(G)$ leaves and $s(G)$ support vertices such that $\gamma_{o}(G)=\frac{n-l(G)+s(G)}{3}$. Then for every $\gamma_{o}(G)$-set $D$, every vertex of $C-D$ has degree two.
\end{theorem}

\begin{proof}
Let $D$ be a $\gamma_{o}(G)$-set and let $y$ be a vertex of $C-D$ such that $deg_{G}(y)\geq 3$. Suppose that $x,z$ are the neighbors of $y$ in $C$ and let $x^{\prime}$ be the other neighbor of $x$ in $C$. By the proof of Theorem \ref{them3}, we have both $x,z$ are in $D$ and $x^{\prime}\notin D$. So, we are in the case $1.2$ of the proof of Theorem \ref{thm2}, as $x\in D, y\not\in D$ and $deg_{G}(y)\geq 3$. We know then that if there exists a vertex $w$ neighbor of $y$ such that $w\not\in D$ then $w\neq z$ as $z\in D$. By deleting the edge $wy$, we get $\gamma_{o}(G)>\frac{n-l(G)+s(G)}{3}$ which is a contradiction. Thus, all neighbors of $y$ are in $D$. Let $G^{\prime}=G-\{xy\}$.
If $deg_{G}(x)=2$, $G^{\prime}$ is a tree of order $n(G^{\prime})=n^{\prime}=n$ with $l(G^{\prime})=l^{\prime}=l(G)+1$ and $s(G^{\prime})=s^{\prime}=s(G)+1$. The set $(D-\{x\})\cup \{x^{\prime}\}$ is a \emph{GOA} of $G^{\prime}$ and so $\gamma_{o}(G^{\prime})\leq |D|=\frac{n-l(G)+s(G)}{3}$. On the other hand, by Theorem \ref{thm1}, we have $\gamma_{o}(G^{\prime})\geq \frac{n^{\prime}-l^{\prime}+s^{\prime}+1}{3}=\frac{n-l(G)-1+s(G)+1+1}{3}=\frac{n-l(G)+s(G)+1}{3}$. We obtain then $\frac{n-l(G)+s(G)+1}{3}\leq \gamma_{o}(G^{\prime})\leq \frac{n-l(G)+s(G)}{3}$ which is a contradiction.

Assume now that $deg_{G}(x)\geq 3$. $G^{\prime}$ is a tree of order $n(G^{\prime})=n^{\prime}=n$ with $l(G^{\prime})=l^{\prime}=l(G)$ and $s(G^{\prime})=s^{\prime}=s(G)$. The set $D$ is a \emph{GOA} of $G^{\prime}$ and so $\gamma_{o}(G^{\prime})\leq |D|=\frac{n-l(G)+s(G)}{3}$. On the other hand, by Theorem \ref{thm1}, we have $\gamma_{o}(G^{\prime})\geq \frac{n^{\prime}-l^{\prime}+s^{\prime}+1}{3}=\frac{n-l(G)+s(G)+1}{3}$ which is a contradiction.\end{proof}

\begin{theorem}\label{thm4} Let $G$ be a connected unicycle graph of order $n$ with cycle $C$, $l(G)$ leaves and $s(G)$ support vertices such that $\gamma_{o}(G)=\frac{n-l(G)+s(G)}{3}$. Then for every $\gamma_{o}(G)$-set $D$, every vertex of $C\cap D$ is a support vertex.
\end{theorem}

\begin{proof}
Let $D$ be a $\gamma_{o}(G)$-set and assume that there exists a vertex $x$ of $C\cap D$ which is not a support vertex. Let $y,z$ be the neighbors of $x$ in $C$. By Theorems \ref{them3}, \ref{thm3}, we know that both $y,z$ are not in $D$ and have degree two. So, $y,z$ are not a support vertices. Without loss of generality, we assume that the other neighbors of $z$ and $y$ in $C$ are a support vertices. If $deg_{G}(x)=2$, let us remove the edge $xy$ from $G$, the resulting graph $G^{\prime}$ is tree of order $n(G^{\prime})=n^{\prime}=n$ with $l(G^{\prime})=l^{\prime}=l(G)+2$ and $s(G^{\prime})=s^{\prime}=s(G)+1$. The set $(D-\{x\})\cup \{z\}$ is a \emph{GOA} of $G^{\prime}$ and so $\gamma_{o}(G^{\prime})\leq |D|=\frac{n-l(G)+s(G)}{3}$. On the other hand, Theorem \ref{thm1} yields $\gamma_{o}(G^{\prime})\geq \frac{n^{\prime}-l^{\prime}+s^{\prime}+1}{3}=\frac{n-l(G)-2+s(G)+1+1}{3}=\frac{n-l(G)+s(G)}{3}$. We obtain then $\gamma_{o}(G^{\prime})=\frac{n-l(G)+s(G)}{3}=\frac{n^{\prime}-l^{\prime}+s^{\prime}+1}{3}$. Thus, in view of Theorem \ref{thm1}, $G^{\prime}\in \mathcal{F}$. Or, $G^{\prime}$ has two support vertices $z$ and its neighbor which are adjacent, or this is impossible by the construction of $\mathcal{F}$. So $deg_{G}(x)\geq 3$.

\begin{claim}\label{cl1}
$x$ has no support vertex as a neighbor.
\end{claim}

\begin{proof} Assume that $x$ is adjacent to a support vertex $w$. By the proof of Theorem \ref{them3}, $w\not\in C$. By removing from $G$, the edge $xw$, the resulting graph $G'$ has two connected components, $G_{0}$ which contains a cycle $C$ and the tree $T'$ which contains $w$. $G_{0}$ is a unicycle graph of order $n(G_{0})=n_{0}$ with $l(G_{0})=l_{0}, s(G_{0})=s_{0}$. $T'$ has order $n(T')=n'$ with $l(T')=l', s(T')=s'$. The sets $D\cap V(G_{0}), D\cap V(T')$ are both a \emph{GOA} of $G_{0}$ and $T'$ respectively. We have $n=n_{0}+n', l(G)=l_{0}+l'$ and $s(G)=s_{0}+s'$. We get by the Theorems \ref{thm1}, \ref{thm2}
\begin{align*}
\gamma_{o}(G)  &  \geq \gamma_{o}(G_{0})+\gamma_{o}(T')\\
&  \geq \frac{n_{0}-l_{0}+s_{0}}{3}+\frac{n'-l'+s'+1}{3}\\
&  =\frac{n-l(G)+s(G)+1}{3}.
\end{align*}
Which is a contradiction.\end{proof}

\begin{claim} \label{cl2}
Every neighbor of $x$ in $V(G)-C$  has degree at most $2$.
\end{claim}

\begin{proof}
Let $y$ be a neighbor of $x$ in $V(G)-C$ such that $deg_{G}(y)\geq 3$. By Claim \ref{cl1}, $y$ is not a support vertex. Furthermore, $y\not\in D$, otherwise, by removing the edge $xy$, we proceed with the same manner as in the proof of Claim \ref{cl1} and we get a contradiction.
If $y$ has all its neighbors in $D$, by removing the edge $xy$, we obtain a graph with two connected components and we get the same contradiction as before. Thus, $y$ has at least a neighbor in $V(G)-(C\cap D)$, say $w$. In this case, by removing the edge $yw$ and by the similar argument to that of the previous case we obtain a contradiction.\end{proof}

\bigskip

Let $y$ be the neighbor of $x$ in $V(G)-C$ such that $deg_{G}(y)=2$ and let $w$ the other neighbor of $y$ which is different from $x$.

\begin{claim} \label{cl3}
Every vertex $w\not\in C$ of distance $2$ from $x$ is a support vertex.
\end{claim}

\begin{proof} Assume that $w$ is not a support vertex. As $x\in D, y\not\in D$ then $w\in D$. We delete then the edge $xy$, the resulting graph has two connected components $G_{0}$ which contains a cycle $C$ and the tree $T'$ which contains $w$. $G_{0}$ is a unicycle graph of order $n(G_{0})=n_{0}$ with $l(G_{0})=l_{0}, s(G_{0})=s_{0}$. $T'$ has order $n(T')=n'$ with $l(T')=l', s(T')=s'$. The sets $D\cap V(G_{0}), D\cap V(T')$ are both a \emph{GOA} of $G_{0}$ and $T'$ respectively. We have $n=n_{0}+n', l(G)=l_{0}+l'-1$ and $s(G)=s_{0}+s'-1$. We get by the Theorems \ref{thm1}, \ref{thm2}
\begin{align*}
\gamma_{o}(G)  &  \geq \gamma_{o}(G_{0})+\gamma_{o}(T')\\
&  \geq \frac{n_{0}-l_{0}+s_{0}}{3}+\frac{n'-l'+s'+1}{3}\\
&  =\frac{n-l(G)+s(G)+1}{3}.
\end{align*}
Which is a contradiction.\end{proof}

\bigskip

Now, if $x$ has all its neighbors in $V(G)-C$ of degree $2$. Let $k=deg_{G}(x)-2$ and let $\{y_{i}, i=1,\ldots, k\}$ be the set of these neighbors. For $i=1,\ldots,k$, let $w_{i}$ the other neighbor of $y_i$ which is different from $x$. We remove then from $G$, the set $\{y_{i}w_{i}, i=1,\ldots, k\}$. We get the components $\cup_{i=1}^{k}T_{i}, G_{0}$, where $T_{i}$ is a tree which contains $w_{i}$, of order $n_{i}$ with $l(T_{i})=l_{i}, s(T_{i})=s_{i}$ and $G_{0}$ the component which contains $C$ which is a unicycle graph of order $n_{0}$ with $l(G_{0})=l_{0}, s(G_{0})=s_{0}$. Set $T'= \cup_{i=1}^{k}T_{i}$. Clearly, $n(T')=n'=\sum_{i=1}^{k}n_{i}, l(T')=l'=\sum_{i=1}^{k}l_{i}$ and $s(T')=s'=\sum_{i=1}^{k}s_{i}$. By the Claims \ref{cl2}, \ref{cl3}, as $y_i$ has degree two and $w_i$ is a support vertex for $i=1,\ldots,k$, it follows that $l(G)=l_{0}+l'-k, s(G)=s_{0}+s'-1$ and $n=n_{0}+n'$. The sets $D\cap V(G_{0}), D\cap V(T_{i})$ for $i=1,\ldots,k$ are a \emph{GOA} of $G_{0}$ and $T_{i}$ respectively. By the Theorems \ref{thm1}, \ref{thm2}, we obtain
\begin{align*}
\gamma_{o}(G)  &  \geq \gamma_{o}(G_{0})+\gamma_{o}(T')\\
&  \geq \frac{n_{0}-l_{0}+s_{0}}{3}+\frac{n'-l'+s'+k}{3}\\
&  =\frac{n-l(G)-k+s(G)+1+k}{3}\\
&  =\frac{n-l(G)+s(G)+1}{3}\\
\end{align*}
Which is a contradiction, so there exists a neighbor of $x_{0}$ of degree one. Thus, $x_{0}$ is a support vertex.\end{proof}

\bigskip

Next, we are interested in characterizing all unicycle graphs that attain the
bound in Theorem \ref{thm2}. For this purpose, we introduce the family $\mathcal{G}$
of unicycle graphs that can be obtained from a sequence $G_{1},G_{2},..,G_{k}, (k\geq1)$
of unicycle graphs, where $G=G_{k}$ and $G_{1}$ is the graph obtained from an even cycle
of order $k$ with vertices $x_{1},x_{2},...,x_{k}$ in this order by adding a
$\frac{k}{2}$ edges such that exactly one of them is incident to $x_{2i-1}$ for each $i$
in $\{1,..,\frac{k}{2}\},$ and if $k\geq2,$ $G_{i+1}$ is obtained recursively from $G_{i}$
by one of the two operations defined below.

\begin{itemize}
\item Operation O$1$: Attach a vertex by joining it by an edge to any support vertex of
$G_{i}$.

\item Operation O$2$: Attach a path $P_{3}=abc$ by joining $a$ by an edge to any support
vertex of $G_{i}.$
\end{itemize}

\bigskip 
Before to prove that the family $\mathcal{G}$ contains all unicycles graph attaining the lower bound of Theorem \ref{thm2}, we need to prove that the following family $\mathcal{G}_{0}$
which is included in $\mathcal{G}$ reaches the bound of Theorem \ref{thm2}.

\bigskip

Let $\mathcal{G}_{0}$ be a family of graphs that can be obtained from $r\geq1$
disjoint stars by first adding $r-1$ edges so that they are incident only with
centers of the stars and then join one center vertex of one star by an edge $u$
to one support vertex of $G_{1}$ and then subdividing $u$ and the $r-1$ edges that connect
the centers exactly once (see Figure \ref{fig}). It is clear to see that every graph in $\mathcal{G}_{0}$ is in $\mathcal{G}$.

\begin{figure}[htbp]
\begin{center}

\scalebox{1.5}
{
\begin{pspicture}(0,-1.16)(6.38,1.16)
\psdots[dotsize=0.12](5.92,0.69)
\psdots[dotsize=0.12](6.3,0.29)
\psdots[dotsize=0.12](6.1,0.29)
\psdots[dotsize=0.12](5.52,0.29)
\psline[linewidth=0.02cm](5.92,0.69)(6.3,0.29)
\psline[linewidth=0.02cm](5.9,0.71)(6.08,0.31)
\psline[linewidth=0.02cm](5.92,0.69)(5.5,0.27)
\psline[linewidth=0.02cm,linestyle=dashed,dash=0.16cm 0.16cm](5.48,0.29)(6.12,0.29)
\psdots[dotsize=0.12](3.38,0.69)
\psdots[dotsize=0.12](3.76,0.29)
\psdots[dotsize=0.12](3.56,0.29)
\psdots[dotsize=0.12](2.98,0.29)
\psline[linewidth=0.02cm](3.38,0.69)(3.76,0.29)
\psline[linewidth=0.02cm](3.36,0.71)(3.54,0.31)
\psline[linewidth=0.02cm](3.38,0.69)(2.96,0.27)
\psline[linewidth=0.02cm,linestyle=dashed,dash=0.16cm 0.16cm](2.94,0.29)(3.58,0.29)
\psdots[dotsize=0.12](2.32,0.69)
\psdots[dotsize=0.12](2.7,0.29)
\psdots[dotsize=0.12](2.5,0.29)
\psdots[dotsize=0.12](1.92,0.29)
\psline[linewidth=0.02cm](2.32,0.69)(2.7,0.29)
\psline[linewidth=0.02cm](2.3,0.71)(2.48,0.31)
\psline[linewidth=0.02cm](2.32,0.69)(1.9,0.27)
\psline[linewidth=0.02cm,linestyle=dashed,dash=0.16cm 0.16cm](1.88,0.29)(2.52,0.29)
\psdots[dotsize=0.12](2.84,0.71)
\psdots[dotsize=0.12](5.52,0.69)
\psline[linewidth=0.02cm,linestyle=dashed,dash=0.16cm 0.16cm](3.4,0.69)(5.5,0.69)
\psline[linewidth=0.02cm](2.32,0.69)(3.38,0.69)
\psline[linewidth=0.02cm](5.54,0.69)(5.96,0.69)
\psline[linewidth=0.02cm](1.1,0.69)(2.32,0.69)
\psdots[dotsize=0.12](1.1,0.69)
\psdots[dotsize=0.12](0.72,1.09)
\psdots[dotsize=0.12](0.3,0.69)
\psline[linewidth=0.02cm](0.72,1.11)(1.1,0.67)
\psline[linewidth=0.02cm](0.72,1.09)(0.3,0.71)
\psdots[dotsize=0.12](0.3,-0.69)
\psdots[dotsize=0.12](1.1,-0.69)
\psdots[dotsize=0.12](1.68,0.69)
\psdots[dotsize=0.12](0.7,-1.09)
\psline[linewidth=0.02cm](1.32,1.09)(1.12,0.71)
\psline[linewidth=0.02cm](0.04,1.09)(0.32,0.69)
\psdots[dotsize=0.12](0.06,1.07)
\psdots[dotsize=0.12](1.3,1.07)
\psline[linewidth=0.02cm](1.08,-0.69)(0.72,-1.09)
\psline[linewidth=0.02cm](0.32,-0.69)(0.7,-1.07)
\psline[linewidth=0.02cm](1.12,-0.69)(1.3,-1.07)
\psline[linewidth=0.02cm](0.28,-0.69)(0.06,-1.05)
\psdots[dotsize=0.12](0.06,-1.07)
\psdots[dotsize=0.12](1.28,-1.07)
\psdots[dotsize=0.12](0.3,0.45)
\psdots[dotsize=0.12](1.1,0.47)
\psdots[dotsize=0.12](0.3,-0.49)
\psdots[dotsize=0.12](1.1,-0.49)
\psline[linewidth=0.02cm](0.3,0.71)(0.3,0.47)
\psline[linewidth=0.02cm](1.1,0.69)(1.1,0.47)
\psline[linewidth=0.02cm](0.3,-0.49)(0.3,-0.69)
\psline[linewidth=0.02cm](1.1,-0.47)(1.1,-0.63)
\psline[linewidth=0.02cm,linestyle=dashed,dash=0.16cm 0.16cm](0.3,0.47)(0.3,-0.49)
\psline[linewidth=0.02cm,linestyle=dashed,dash=0.16cm 0.16cm](1.1,0.47)(1.1,-0.45)
\usefont{T1}{ptm}{m}{n}
\rput(2.3601563,0.875){\tiny $x_1$}
\usefont{T1}{ptm}{m}{n}
\rput(5.9559374,0.875){\tiny $x_r$}
\end{pspicture}
}

\end{center}
\caption{The family $\mathcal{G}_{0}$.} \label{fig}
\end{figure}
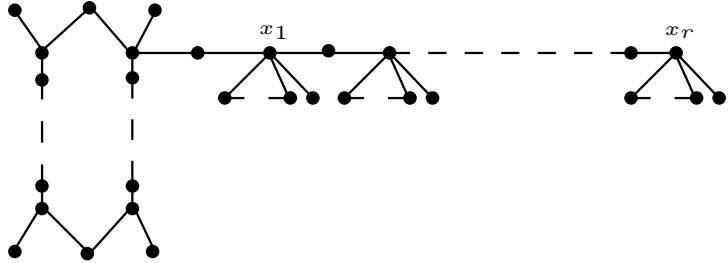

\begin{proposition}
\label{prop0}Let $G\in\mathcal{G}_{0}$ be a graph of order $n$ with $\l(G)$
leaves and $s(G)$ support vertices. Then $\gamma_{o}(G)=\frac{n-l(G)+s(G)}{3}.$
\end{proposition}

\begin{proof}
Let $C$ be the unique even cycle of $G$ and let $L_{c},S_{c}$ be the number of
leaves and support vertices of $C$, respectively. Let $T$ and $G\diagdown T$
be the two connected components of $G-xy$ where $x$ is a support vertex of $C$
and $y\notin C$ is a neighbor of $x$ of degree $2$ in $G$. Clearly, $T$ is a tree
of order $n(T)$ that belongs to $\mathcal{F}.$ Therefore\ $n(T)=n-(L_{c}%
+2S_{c}),$ $l(T)=\ell(G)-L_{c}+1$ and $s(T)=s(G)-S_{c},$ Then by Observation
\ref{obs00},
\[
\gamma_{o}(G)=\gamma_{o}(T)+\gamma_{o}(G\diagdown T)=\gamma_{o}(T)+S_{c}.
\]
As $T$ is a tree of $%
\mathcal{F}%
,$ it follows that $\gamma_{o}(T)=\frac{n(T)-l(T)+s(T)+1}{3}.$ Hence
\begin{align*}
\gamma_{o}(G)  &  =\dfrac{n(T)-l(T)+s(T)+1}{3}+S_{c}\\
&  =\dfrac{n-(L_{c}+2S_{c})-\left(  \ell-L_{c}+1\right)  +1+s-S_{c}+3S_{c}}%
{3}\\
&  =\frac{n-\ell(G)+s(G)}{3}.\end{align*}\end{proof}

\begin{theorem}
Let $G$ be a connected unicycle graph of order $n$ with $l(G)$ leaves and $s(G)$
support vertices. Then $\gamma_{o}(G)=\frac{n-l(G)+s(G)}{3}$ if and only if $G\in\mathcal{G}.$
\end{theorem}

\begin{proof}
Let $G\in\mathcal{G}$ and $C_{k}$ be the unique even cycle of $G$ of order
$k$. Remove from $G$, $\frac{k}{2}$ edges $x_{i}y_{i}$ ($1\leq i\leq\frac
{k}{2})$ of $G$ such $x_{i}$ is a support vertex of $C_{k},$ $y_{i}\notin
C_{k}$ is a neighbor of $x_{i}$ of degree $2$ in $G$.

Let $T_{1},T_{2},..,$ $T_{\frac{k}{2}}$ and $H_{1}$ be the connected components of the resulting graph. $H_{1}$ is the
component that contains the cycle $C_{k}$ and $T_{i}$ is a tree. Observe that for each
$i\in\{1,...,\frac{k}{2}\},$ $T_{i}$ is in $\mathcal{F}$. Thus by Theorem
\ref{thm1},
\begin{equation}
\gamma_{o}(T_{i})=\frac{n(T_{i})-l(T_{i})+s(T_{i})+1}{3}\text{ for each }%
i\in\{1,...,\frac{k}{2}\}. \label{cond0}%
\end{equation}
Set $G_{0}=$ $G\diagdown\cup_{i=2}^{i=\frac{k}{2}}T_{i}$. Then $G_{0}$ is the
subgraph induced by the vertices of $H_{1}$ and $T_{1}$. Clearly, $G_{0}%
\in\mathcal{G}_{0}.$ Therefore\ by Observation \ref{obs00}, equality (\ref{cond0}),
and Proposition \ref{prop0}, we get
\begin{align*}
\gamma_{o}(G)  &  =\gamma_{o}(G_{0})+\sum_{i=2}^{i=\frac{k}{2}}\gamma
_{o}(T_{i})\\
&  =\frac{n(G_{0})-l(G_{0})+s(G_{0})}{3}+\sum_{i=2}^{\frac{k}{2}}\frac
{n(T_{i})-l(T_{i})+s(T_{i})+1}{3}\\
&  =\frac{n(G_{0})-l(G_{0})+s(G_{0})}{3}+\sum_{i=2}^{\frac{k}{2}}\frac
{n(T_{i})-l(T_{i})+s(T_{i})}{3}+\frac{\frac{k}{2}-1}{3}%
\end{align*}

As $n=n(G_{0})+\sum_{i=2}^{i=\frac{k}{2}}n(T_{i}),$ $s(G)=s(G_{0}%
)+\sum_{i=2}^{i=\frac{k}{2}}s(T_{i})$ and $l(G)=l(G_{0})+\sum_{i=2}%
^{i=\frac{k}{2}}l(T_{i})-(\frac{k}{2}-1),$ it follows that
\[
\gamma_{o}(G)=\frac{n-l(G)+s(G)}{3}.
\]

\bigskip

Conversely, Let $D$ be a $\gamma_{o}(G)$-set and let $x$ be a vertex in $C\cap D$. By Theorem \ref{thm4}, $x$ is a support vertex. Let $y$ be a neighbor of $x$ in $V(G)-C$ of degree $2$, such a vertex exists by Claim \ref{cl2}. Suppose that $w$ is the other neighbor of $y$ in $V(G)-C$. It follows by Claim \ref{cl3}, that $w$ is a support vertex. Let us remove from $G$ the edge $xy$, the resulting graph has two connected components graphs $G^{\prime\prime}$ which is a unicycle graph containing $x$, and $T'$ which a tree that contains $y$.
$G^{\prime\prime}$ is a unicycle graph with cycle $C$ which is an even cycle by Theorem \ref{them3}. Let $S_{c}, L_{c}$ be the number of support vertices and leaves of $C$, respectively.
As every vertex of $C\cap D$ is a support vertex (by Theorem \ref{thm4}) and any vertex of $C-D$ has degree two by the Theorem \ref{thm3}, thus if $S_{c}=L_{c}$, then $G^{\prime\prime}=G_{1}\in \mathcal{G}$, otherwise, $G^{\prime\prime}$ can be obtained from $G_{1}$ by the operation $O1$ and so $G^{\prime\prime}\in \mathcal{G}$. Let $G^{\prime\prime}$ be of order $n^{\prime\prime}$ with $l(G^{\prime\prime})=l^{\prime\prime}, s(G^{\prime\prime})=s^{\prime\prime}$ and let $T'$ be of order $n'$ with $l(T')=l', s(T')=s'$. As $G^{\prime\prime}\in \mathcal{G}$, it follows then from the necessary condition that $\gamma_{o}(G^{\prime\prime})=\frac{n^{\prime\prime}-l^{\prime\prime}+s^{\prime\prime}}{3}$. The sets $D\cap V(G^{\prime\prime})$ and $D\cap V(T')$ are both a \emph{GOA} of $G^{\prime\prime}$ and $T'$ respectively, then
\begin{equation}\label{eq1}
\gamma_{o}(G)\geq \gamma_{o}(G^{\prime\prime})+\gamma_{o}(T').
\end{equation}

On the other hand, as both $x,w$ are a support vertices and $y$ is of degree two, then $n=n^{\prime\prime}+n', l(G)=l^{\prime\prime}+l'-1, s(G)=s^{\prime\prime}+s'$ and
any $\gamma_{o}(G^{\prime\prime})$-set $\cup \gamma_{o}(T')$-set is a \emph{GOA} of $G$, and so,

\begin{equation}\label{eq2}
\gamma_{o}(G)\leq \gamma_{o}(G^{\prime\prime})+\gamma_{o}(T').
\end{equation}

By the inequalities (\ref{eq1}), (\ref{eq2}), $\gamma_{o}(G) = \gamma_{o}(G^{\prime\prime})+\gamma_{o}(T')$. So

\begin{align*}
\gamma_{o}(T')  &  =\gamma_{o}(G)- \gamma_{o}(G^{\prime\prime})\\
&  =\frac{n-l(G)+s(G)}{3}-\frac{n^{\prime\prime}-l^{\prime\prime}+s^{\prime\prime}}{3}\\
&  =\frac{(n-n^{\prime\prime})-(l(G)-l^{\prime\prime})+(s(G)-s^{\prime\prime})}{3}
&  =\frac{n'-(l'-1)+s'}{3}=\frac{n'-l'+s'+1}{3}.
\end{align*}

Thus, $T'\in \mathcal{F}$. So, if $G^{\prime\prime}= \mathcal{G}_{1}$, then $G= \mathcal{G}_{0}\in \mathcal{G}$, else, $G$ can be obtained from $\mathcal{G}_{0}$ by the Operation $O1$ and then $G\in \mathcal{F}$ which completes the proof.\end{proof}

\bigskip

\section*{Acknowledgements}
The authors thank Professor M. Chellali for his helpful suggestions.

\end{document}